\documentclass[12pt]{article}

\usepackage[notref,notcite]{showkeys} % mostrar os labels no dvi...
\usepackage{amsmath}
\usepackage{amsthm} % \begin{proof} \end{proof}
\usepackage{amsfonts}
\usepackage{latexsym}
\usepackage{amssymb}  % para fazer IR, IN, etc..
\usepackage{epsfig}

\usepackage{epsfig}

\usepackage{url}

\newcommand{\R}{\mathbb{R}}
\newcommand{\N}{\mathbb{N}}

\newcommand{\calP}{\mathcal{P}}

\DeclareMathOperator{\dom}{dom}

\DeclareMathOperator{\range}{R}
\DeclareMathOperator{\graph}{Gr}

\newcommand{\inner}[2]{\langle{#1},{#2}\rangle}

\newcommand{\tos}{\rightrightarrows} % point-to-set mappings

\newtheorem{theorem}{Theorem}[section]
\newtheorem{lemma}[theorem]{Lemma}
\newtheorem{corollary}[theorem]{Corollary}

\newtheorem{proposition}[theorem]{Proposition}

\newtheorem{definition}[theorem]{Definition}

\renewcommand{\hat}{\widehat}
\renewcommand{\tilde}{\widetilde}

\title{A non-type (D) operator in $c_0$}

\author{Orestes Bueno\footnote{Instituto de Mat\'ematica Pura e Aplicada 
(IMPA), Estrada Dona
Castorina 110, Rio de Janeiro, RJ, CEP 22460-320, Brazil,
{\tt obueno@impa.br}.
The work of this author was partially supported by CAPES}
\qquad B. F. Svaiter\footnote{Instituto de Mat\'ematica Pura e Aplicada, 
(IMPA), Estrada Dona
Castorina 110, Rio de Janeiro, RJ, CEP 22460-320, Brazil,
{\tt benar@impa.br}.
The work of this author was partially supported by
 CNPq grants no. 
474944/2010-7, 303583/2008-8 and  FAPERJ grant E-26/110.821/2008.
}}

\begin{document}

\maketitle

%\begin{center}
% {\it Dedicated to Professor J. M. Borwein on the occasion of his 60th birthday}
%\end{center}

\begin{abstract}
  Previous examples of non-type (D) maximal monotone operators were
  restricted to $\ell^1$, $L^1$, and Banach spaces containing isometric
  copies of these spaces.  This fact led to the conjecture that
  non-type (D) operators were restricted to this class of Banach spaces.
  We present a linear non-type (D) operator in $c_0$.\\

keywords: maximal monotone, type (D), Banach space, extension, bidual. 
\end{abstract}

\section{Introduction}
Let $U$, $V$ arbitrary sets. A \emph{point-to-set} (or multivalued) operator
$T:U\tos V$ is a map $T:U\to \calP(V)$, where $\calP(V)$ is the power set
 of $V$. Given 
$T:U\tos V$, the \emph{graph} of $T$ is the set
\[
\graph(T):=\{(u,v)\in U\times V\:|\:v\in T(u)\},
\]
the  \emph{domain} and the \emph{range} of $T$ are, respectively,
\[
\dom(T):=\{u\in U\:|\:T(u)\neq\emptyset\},\qquad
\range(T):=\{v\in V\:|\:\exists u\in U, \;v\in T(u) \}
\]
and the \emph{inverse} of $T$ is the  point-to-set operator $T^{-1}:V\tos U$,
\[
 T^{-1}(v)=\{u\in U\:|\:v\in T(u)\}.
\]
A point-to-set operator $T:U\tos V$ is called \emph{point-to-point} if for
every $u\in\dom(T)$, $T(u)$ has only one element.
%and in this case  we use the notation $T:U\to V$ (which \emph{does not} mean  that $\dom(T)=U$).
Trivially, a point-to-point operator is injective if, and only if,
its inverse is also point-to-point.

Let $X$ be a real Banach space.
We 
use the notation $X^*$ for the topological
dual of $X$.
From now on  $X$ is identified with its canonical injection into $X^{**}=(X^*)^*$
and the duality product in $X\times X^*$ will be denoted by  $\inner{\cdot}{\cdot}$,
%We will identify $X$ with its canonical injection into $X^{**}=(X^*)^*$
%and we will use the notation $\inner{\cdot}{\cdot}$ for the duality product,
\[
\inner{x}{x^*}=\inner{x^*}{x}=x^*(x),\qquad x\in X,x^*\in X^*.
\]
A point-to-set operator $T:X\tos X^*$ (respectively $T:X^{**}\tos X^*$) is \emph{monotone}, if
\[
\inner{x-y}{x^*-y^*}\geq 0,\quad \forall (x,x^*),(y,y^*)\in \graph(T),
\]
(resp. $\inner{x^*-y^*}{x^{**}-y^{**}}\geq 0$, $\forall (x^{**},x^*),(y^{**},y^*)\in \graph(T)$), and it is \emph{maximal monotone} if it is monotone and maximal in the family of monotone operators in $X\times X^*$ (resp. $X^{**}\times X^*$)
with respect to the order of inclusion of the graphs.

We denote $c_0$ as the space of real sequences converging to $0$ and $\ell^{\infty}$ as the space of real bounded sequences, both endowed with the $\sup$-norm
\[
\|(x_k)_k\|_{\infty}=\sup_{k\in\N} |x_k|,
\]
and $\ell^1$ as the space of absolutely summable real sequences, endowed with the $1$-norm,
\[
\|(x_k)_k\|_{1}=\sum_{k=1}^{\infty} |x_k|.
\]
The dual of $c_0$ is identified with  $\ell^1$ in the following sense: for $y\in\ell^1$
\[
y(x)=\inner{x}{y}=\sum_{i=1}^\infty x_iy_i, \qquad \forall
x\in c_0.
\]
Likewise, the dual of $\ell^1$ is identified with $\ell^\infty$. It is well known that $c_0$ (as well as $\ell^1$, $\ell^{\infty}$, etc.) is a non-reflexive Banach space.

Let $X$ be a \emph{non-reflexive} real Banach space and $T:X\tos X^*$ be
maximal monotone.
Since $X\subset X^{**}$, the point-to-set operator $T$ can
also be regarded as an operator from $X^{**}$ to $X^*$. We denote $\hat{T}:X^{**}\tos X^*$ as the operator such that
\[
\graph(\hat{T})=\graph(T).
\]
If $T:X\tos X^*$ is maximal monotone then $\hat{T}$ is (still) trivially monotone but, in general, not maximal monotone. 
%
%If $X$ is non-reflexive, then $T$ in general is not
%maximal monotone in $X^{**}\times X^*$.
%
Direct use of  the Zorn's Lemma shows that $\widehat{T}$ has a 
maximal monotone extension.
%
% maximal monotone such that $T\subset \hat{T}$.
So it is natural to ask if such maximal monotone extension to the bidual is
unique. Gossez \cite{JPGos0,JPGos1,JPGos3,JPGos2} gave a sufficient condition for uniqueness of
such an extension.

\begin{definition}[\cite{JPGos0}]
  Gossez's \emph{monotone closure} (with respect to $X^{**}\times X^*$) 
of a maximal monotone operator $T:X\tos X^*$, is 
the point-to-set operator $\widetilde T:X^{**}\tos X^*$ whose graph
$\graph\left(\widetilde T\right)$ is
given by
\[
\graph\left(\widetilde T\right) = \{(x^{**},x^*)\in X^{**}\times X^*\:|\:\inner{x^*-y^*}{x^{**}-y}\geq 0,\,\forall (y, y^*) \in T\}. 
\]
A maximal monotone operator $T : X\tos X^* $, is of \emph{Gossez type (D)} if for any
$(x^{**},x^*) \in \graph\left( \widetilde T \right)$ , there exists a
bounded net $\bigg((x_i, x^*_i)\bigg)_{i\in I}$ in $\graph(T)$
which converges to $(x^{**}, x^*)$ in the $\sigma(X^{**},
X^*)\times$strong topology of $X^{**}\times X^*$.
\end{definition}

Gossez proved~\cite{JPGos2} that a maximal monotone operator $T:X\tos X^*$ of
of type (D) has unique maximal monotone extension to the 
bidual, namely, its Gossez's monotone closure $\widetilde T:X^{**}\tos X^*$.
Beside this fact, maximal monotone operators of type (D) share 
many properties with maximal monotone operators defined in \emph{reflexive} Banach
spaces as, for example, convexity of the closure of the domain and convexity
of the closure of the range \cite{JPGos0}.

Gossez gave an example of a non-type (D) operator on $\ell^1$~\cite{JPGos1}.
Later, Fitzpatrick and Phelps gave an example of a non-type (D) on
$L^1[0,1]$~\cite{FitPhe95}.
% Recently, the authors gave another example of
% non-type (D) operator on $\ell^1$~\cite{SvaBue}.
%
%These examples led to the conjecture~\cite[\S 4]{Bor10} that if a Banach space
%doesn't contain an isometric copy of $\ell^1$, then every maximal monotone
%operator in this space is of type (D). In this work, we answer negatively such
%conjecture by giving an example of a non-type (D) operator on $c_0$ and proving
%that, for every space which contains a isometric copy of $c_0$, a non-type (D)
%operator can be defined.
%
These examples led  Professor J. M. Borwein to
define  Banach spaces of type (D) as those 
Banach spaces where every maximal monotone operator is of type (D), and
to formulate the following most interesting conjecture~\cite[\S 4, question 3]{Bor10}:
\begin{quotation}
  \noindent $\bullet$
Are any nonreflexive spaces $X$ of type (D)? That is, are there nonreflexive spaces
on which all maximal monotones on $X$ are type (D). I conjecture `weakly' that if
$X$ contains no copy of $\ell^1(\N)$   then $X$ is type (D) as would hold in $X = c_0$.
\end{quotation}
In this work, we answer negatively such
conjecture by giving an example of a non-type (D) operator on $c_0$ and proving
that, for every space which contains a isometric copy of $c_0$, a non-type (D)
operator can be defined.

\section{A non-type (D) operator on $c_0$}

Gossez's operator~\cite{JPGos1} $G:\ell^1\to\ell^\infty$ is defined as
\begin{equation}
  \label{eq:G}
G(x)=y, \qquad y_n=\sum_{i=n+1}^{\infty}x_i-
\sum_{i=1}^{n-1}x_i,
\end{equation}
which is linear,  continuous, anti-symmetric and,
therefore, maximal monotone. 
%In the proof of \cite[Proposition 3.2]{} it was proved that
Gossez's operator is also injective
(see the proof of~\cite[Proposition 3.2]{SvaBue}).
This operator will be used to define a non-type
(D) maximal monotone operator in $c_0$. 

\begin{lemma}\label{lem:1}
  The operator
  \begin{equation}
    \label{eq:op}
    T:c_0\tos\ell^1,\qquad
    T(x)=\{ y\in \ell^1\;|\; -G(y)=x\}
  \end{equation}
  is point-to-point in its domain, 
  is maximal monotone and its range is
  \begin{equation}
    \label{eq:range}
  R(T)=\left\{y\in\ell^1\;\left|\; \sum_{i=1}^{\infty} y_i=0\right.\right\}
  \;.
  \end{equation}
\end{lemma}
\begin{proof}
  Since $G$ is injective, $T$ is point-to-point in its domain. Moreover, direct use
  of \eqref{eq:G} shows that $G$ is linear and $\inner{y}{G(y)}=0$ for any $y\in\ell^1$,
  which proves that $T$ is also linear and monotone.
  
  Using \eqref{eq:G} we conclude that for any $y\in\ell^1$,
  \[
  \lim_{n\to\infty}(G(y))_i=-\sum_{i=1}^\infty y_i
  \]
  which proves \eqref{eq:range}.

  Suppose that $x\in c_0$, $y\in \ell^1$ and 
  \begin{equation}
    \label{eq:mon}
  \inner{x-x'}{y-y'}\geq 0,\qquad \forall (x',y')\in\graph(T).
\end{equation}
Define, $u^1=(-1,1,0,0,\dots)$, $u^2=(0,-1,1,0,0),\dots$, that is
\begin{equation}
  \label{eq:um}
(u^m)_i=
\begin{cases} -1,& i=m\\
  1,& i=m+1\\
  0,&\mbox{otherwise}
\end{cases}
\qquad i=1,2,\dots
\end{equation}
and let 
\begin{equation}
  \label{eq:vm}
v^m=G(u^m), \qquad (v^m)_i=
\begin{cases} 1,& i=m \mbox{ or }i=m+1\\
  0,&\mbox{otherwise}
\end{cases}
\qquad i=1,2,\dots
\end{equation}
where the expression of $(v^m)_i$ follows from \eqref{eq:G} and \eqref{eq:um}.

  Direct use of  \eqref{eq:um}, \eqref{eq:vm} and \eqref{eq:op} shows that
  $T(-\lambda v^m)=\lambda u^m$ for $\lambda\in\R$ and $m=1,2,\dots$. 
  Therefore, for any $\lambda\in \R$, $m=1,2,\dots$
  \[
  \inner{x+\lambda v_m}{y-\lambda u^m}\geq 0
  \]
  which is equivalent to
  \[
  \inner{x}{y} +\lambda [ \inner{ v_m}{y}
    -\inner{x}{u^m}]\geq 0.
  \]
  Since the above inequality holds for any $\lambda$, it holds that  %, we conclude that
  \[
\inner{x}{u^m}=\inner{v^m}{y},\qquad
  m=1,2,\dots
  \]
  which, in view of \eqref{eq:um}, \eqref{eq:vm} is equivalent to
  %Therefore
  \begin{equation}
    x_{m+1}-x_m= y_{m+1}+y_m, \qquad m=1,2,\dots
    \label{eq:ufa}
  \end{equation}
  Adding the above equality for $m=i,i+1,\dots,j$ we conclude that
  \[
  x_{j+1}=x_i+y_i+2\sum_{k=i+1}^j y_k+y_{j+1},\qquad i<j.
  \]
  Using the assumptions $x\in c_0$, $y\in\ell^1$
  and taking the limit $j\to\infty$ in the
  above equation, we conclude that
  \[
  x_i=
  -\left[y_i+2\sum_{k=i+1}^\infty y_k\right]
  = -\left[ G(y)_i+\sum_{k=1}^\infty y_k\right].
  \]
  Also, using the above relation between $x$ and $y$ in \eqref{eq:mon} with
  $x'=0$, $y'=0$ we obtain
  \[
  -\left[
  \sum_{k=1}^\infty y_k \right]^2\geq0.
  \]
  Combining the two above equations we conclude that $x=-G(y)$. Hence 
  $(x,y)\in \graph(T)$, which proves
  the maximal monotonicity of $T$.
\end{proof}

\begin{proposition}\label{pro:nond}
  The operator $T:c_0\tos\ell^1$ defined in Lemma~\ref{lem:1} has infinitely many maximal
  monotone extensions to $\ell^\infty\tos \ell^1$. In particular, $T$ is non-type (D).
\end{proposition}
\begin{proof}
  Let
  \[
  e=(1,1,1,\dots)
  \]
  We claim that 
  \begin{equation}
    \label{eq:mp}
  \inner{-G(y)+\alpha e-x'}{y-y'}=\alpha\inner{y}{e},\qquad
  \forall (x',y')\in \graph(T), y\in \ell^1.
  \end{equation}
  To prove this claim, first
  use \eqref{eq:op} and \eqref{eq:G} to conclude that $x'=-G(y')$ and
  \[
  \inner{-G(y)-x'}{y-y'}=\inner{G(y'-y)}{y-y'}=0
  \]
  As  $y'\in R(T)$,  using \eqref{eq:range} we have 
  $\inner{e}{y'}=0$, which combined with the above equation yields \eqref{eq:mp}.
  
  Take $\tilde y\in\ell^1$ such that $\inner{\tilde y}{e}>0$ and define
  \[ 
  x^\tau=-G(\tau\tilde y)+\frac{1}{\tau}  e, \qquad 0<\tau<\infty.
  \]
  In view of \eqref{eq:mp},
  \[
  \left (\, x^\tau,\tau\tilde y\, \right)\in  \widetilde T,
   \qquad 0<\tau<\infty.
  \]
  Therefore, for each $\tau \in (0,\infty)$ there exists a
  maximal monotone extension $T_\tau:\ell^\infty\tos\ell^1$  of $T$
  such that
  \[ (x^\tau,\tau \tilde y)\in G(T_\tau).
  \]
  However, these extensions are distinct because if $\tau,\tau'\in (0,\infty)$ and
  $\tau\neq\tau'$ then
  \[
  \inner{x^\tau-x^{\tau'}}{\tau \tilde y-\tau'\tilde y}=(\tau-\tau')(1/\tau-1/\tau')
  \inner{\tilde y}{e}<0.
  \]
\end{proof}

\begin{theorem}\label{teo:nond}
Let $X$ be a Banach space such that there exists a non-type (D) maximal monotone operator $T:X\tos X^*$, and let $\Omega$ be another Banach space which contains an isometric copy of $X$. Then there exists a non-type (D) maximal monotone operator $S:\Omega\tos\Omega^*$.
\end{theorem}
\begin{proof}
We can identify $X$ as a closed subspace of $\Omega$. Define $S:\Omega\tos\Omega^*$ as
\[
(w,w^*)\in\graph(S)\quad\Longleftrightarrow\quad w\in X\text{ and }(w,w^*|_X)\in \graph(T).
\]
This implies that for every $v^*\in\Omega^*$ such that $v^*|_X\equiv 0$ then 
\begin{equation}\label{eq:nrm}
(w,w^*+t v^*)\in\graph(S),\qquad\forall \,(w,w^*)\in\graph(S),\,t\in\R,
\end{equation}

Take $(w,w^*),(z,z^*)\in \graph(S)$, then $w,z\in X$ and  $w-z\in X$. Thus, since $T$ is monotone and $(w,w^*|_X),(z,z^*|_X)\in \graph(T)$,
\[
\inner{w-z}{w^*-z^*}=\inner{w-z}{w^*|_X-z^*|_X}\geq 0.
\]

Now we prove that $S$ is maximal monotone. Let $(z,z^*)\in\Omega\times\Omega^*$ such that
\begin{equation}\label{eq:mnrl}
\inner{w-z}{w^*-z^*}\geq 0,\qquad\forall \,(w,w^*)\in\graph(S), 
\end{equation}
and assume that $z\notin X$. Then, by the Hahn-Banach Theorem, there exists $v^*\in\Omega^*$ such that $v|_X\equiv 0$ and $\inner{z}{v^*}=1$.
Using~\eqref{eq:nrm}, we obtain
\[
\inner{w-z}{w^*+tv^*-z^*}\geq 0,\qquad\forall \,(w,w^*)\in\graph(S),\,t\in\R,
\]
which is equivalent to
\[
\inner{w-z}{w^*-z^*}\geq t\inner{z-w}{v^*}=t,\qquad\forall \,(w,w^*)\in\graph(S),\,t\in\R,
\]
which leads to a contradiction if we let $t\to+\infty$. Hence $z\in X$. Now, using~\eqref{eq:mnrl} and the maximal monotonicity of $T$, we conclude that $(z,z^*|_X)\in\graph(T)$. Therefore $(z,z^*)\in \graph(S)$ and $S$ is maximal monotone.

Finally, we use \cite[eq.\ (5) and Theorem 4.4, item 2]{BSMMA-TypeD} to prove that $S$ is non-type (D). As $T$ is non-type (D) on $X\times X^*$ then there exists $(x^*,x^{**})\in X^*\times X^{**}$ such that
\[
\sup_{(y,y^*)\in \graph(T)}\inner{y}{x^*-y^*}+\inner{y^*}{x^{**}}<\inner{x^*}{x^{**}}.
\]
Define $z^{**}:X^*\to\R$ as $\inner{z^*}{z^{**}}=\inner{z^*|_X}{x^{**}}$ and let $z^*$ be any continuous extension of $x^*$ from $X$ to $\Omega$. Take $(w,w^*)\in\graph(S)$, then $w\in X$, $(w,w^*|_X)\in\graph(T)$ and, hence,
\begin{gather*}
\inner{w}{z^*-w^*}=\inner{w}{z^*|_X-w^*|_X}=\inner{w}{x^*-w^*|_X},\\
\inner{w^*}{z^{**}}=\inner{w^*|_X}{x^{**}},\qquad \inner{z^*}{z^{**}}=\inner{z^*|_X}{x^{**}}=\inner{x^*}{x^{**}}.
\end{gather*}
Therefore,
\[
\sup_{(w,w^*)\in \graph(S)}\inner{w}{z^*}+\inner{w^*}{z^{**}}-\inner{w}{w^*}<\inner{z^*}{z^{**}},
\]
and $S$ is non-type (D) on $\Omega\times\Omega^*$.
\end{proof}

Using Proposition~\ref{pro:nond} and Theorem~\ref{teo:nond}, we have the following Corollary.
\begin{corollary}
Any real Banach space $X$ which contains an isometric copy of $c_0$ has a non-type (D) maximal monotone operator.
\end{corollary}

%\bibliographystyle{abbrv}
%\bibliography{refers}
\end{document}